\documentclass{article}

\usepackage{amsmath, amssymb, amsthm}
\newtheorem{theorem}{Theorem}
\newtheorem{definition}{Definition}
\newtheorem{lemma}{Lemma}

\title{Playing impartial games on a simplicial complex as an extension of the emperor sum theory}

\author{Koki Suetsugu \\ National Institute of Informatics}

\begin{document}
\maketitle 

\abstract{
In this paper, we considered impartial games on a simplicial complex.
Each vertex of a given simplicial complex acts as a position of an impartial game. Each player in turn chooses a face of the simplicial complex and, for each position on each vertex of that face, the player can make an arbitrary number of moves. Moreover, the player can make only a single move for each position on each vertex, not on that face.
We show how the $\mathcal{P}$-positions of this game can be characterized using the $\mathcal{P}$-position length. This result can be considered an extension of the emperor sum theory. While the emperor sum only allowed multiple moves for a single component, this study examines the case where multiple moves can be made for multiple components, and clarifies areas that the emperor sum theory did not cover. 
}

\section{Introduction}\label{Sec: Intro}
Combinatorial game theory studies two-player perfect information games with no chance moves. An {\em option} of position $g$ is a position after one move from $g,$ and a {\em strict follower} of $g$ is a position after an arbitrary number of moves from $g$. {\em Impartial games} are games in which both players have the same set of options at each position. In this study, we consider that games are under the {\em normal play} convention; that is, the player who makes the last move is the winner. We also assume that games are {\em loop-free}; that is, a position occurs at most once in a round.
Under these conventions, exactly one of the players has a winning strategy at any given position. We define a position in which the next (resp. previous) player's winning strategy is an {\em $\mathcal{N}$-position} (resp. {\em $\mathcal{P}$-position}). For more details on combinatorial game theory, see \cite{CGT, LIP}.

The {\em disjunctive sum} of games is one of the most popular concepts in combinatorial game theory. For any position $g$ and $h$, the disjunctive sum of impartial games $g+h$ is a game the options of which are $g'+h$ and $g+h'$, where $g'$ and $h'$ include all the options of $g$ and $h$, respectively. That is, in the disjunctive sum of games, a player chooses one component and moves. The disjunctive sum of games is known to characterize the $\mathcal{P}$-position using the parameter $\mathcal{G}$-value (or Sprague-Grundy value) of components, and is often studied as a central topic in combinatorial game theory.

Theories such as graph theory are known to have a variety of applications. By contrast, combinatorial game theory has been minimally used in applied mathematics to date. In terms of applicability, a game-like mechanism would be similar to the security guarantees in various topics such as automated driving and cyber-attacks. To guarantee the safety of something, it must be shown that there are ways to respond to an attack or accident. This is similar to the guarantee of the existence of a winning strategy that can win no matter what move the other player initiates. In this manner, a disjunctive sum situation is unlikely to occur. This is because in a situation where a system is built as a collection of multiple units, it is necessary to assume that only one of them will be attacked. Therefore, it is thought that the possibility of application can be expanded by deepening the research managing situations where multiple components can be launched.
Therefore, this research will enrich our knowledge in such areas by considering new game combination forms. In this study, we consider a situation where not only multiple components can be moved but also numerous moves can be made for some of the components. In our previous work on emperor sum, we were limited to a single component that can be moved multiple times, but in this work, we further generalize and consider the case where such moves can be made on multiple components.

\subsection{Early Results}
{\sc Nim} is among the most well-known impartial games. In this game, there are several stone heaps. A player chooses a heap and removes any positive stones from the heap. 
Since the games are being considered under the normal conventions of play, in {\sc nim}, the player who removes the last stone wins. In \cite{Bou01}, Bouton proved that the position in {\sc nim} is a $\mathcal{P}$-position if, and only if, the bitwise XOR of the numbers of stones in the heaps is zero.

Let $V$ be a finite set of vertices.
A {\em simplicial complex} $\Delta$ on $V$ is a subset of $2^{\mid V \mid}$
 such that for any element $v \in V,$ $\{v\} \in \Delta$, and if $F \in \Delta$ and $G \subseteq F,$ then $G \in \Delta$. An element of $\Delta$ is termed a {\em face} of $\Delta$.
Ehrenborg and Steingr\'{i}msson studied {\sc nim} on simplicial complexes or {\sc simplicial nim} \cite{ES96}. Let $\Delta$ be a simplicial complex of a finite set $V$. Each vertex has several stones. In {\sc nim} on $\Delta$, the players, in turn, choose a non-empty face $F$ in $\Delta$ and arbitrarily remove any positive number of stones from all vertices in $F$. The original {\sc nim}, {\sc moore's nim} \cite{Moo09}, and {\sc circular nim} \cite{DS13} can be considered as special cases of this game.
Ehrenborg and Steingr\'{i}msson did not characterize the $\mathcal{P}$-positions of {\sc nim} on $\Delta$ without restrictions but discovered sound constructions and characterized $\mathcal{P}$-positions for certain cases. Horrocks and Penn also researched this topic \cite{Hor10} and \cite{Pen21}. 

The intention is to generalize Ehrenborg and Steingr\'{i}msson's study to impartial games. 

\begin{definition}
Let $\Delta$ be a simplicial complex on a finite set $V = (v_1, v_2, \ldots, v_n)$ and let $G = (g_1, g_2, \ldots, g_n)$ be a collection of impartial game positions.

{\em Simplicial emperor sum of $G$ on $\Delta$}, which is denoted by $\Delta(G)$, is a position such that on each vertex $v_i$, there is an impartial game position $g_i$ and a player, in turn, chooses a face $F$ in $\Delta$ and makes many arbitrary moves for positions on all vertices in $F$. In addition, for each position on the vertex, and not in $F$, the player makes at most one move.
\end{definition}

To characterize the $\mathcal{P}$-positions of this ruleset, we use the {\em $\mathcal{P}$-position length}.

\begin{definition}
For any position $g$, the $\mathcal{P}$-position length of $g$ is 
$$
\mathrm{Pl}(g) =
\left\{
\begin{array}{ll}
0, & \text{If $g$ is a terminal position.} \\
\max(\{\mathrm{Pl}(g'): g' \text{ is a $\mathcal{P}$-position and strictly follows } g\}) + 1,     & \text{Otherwise.}
\end{array}
\right.
$$
\end{definition}
The $\mathcal{P}$-position length is used to characterize the $\mathcal{P}$-positions of the {\em emperor sum} of the games.

The following definition and theorem of emperor sum are introduced in \cite{Sue21}.

\begin{definition}
Let $g_1, g_2, \ldots, g_n$ be positions in impartial games.
The emperor sum of the positions $\mathcal{E}(g_1, g_2, \ldots, g_n)$ is a position with the options $\mathcal{E}(g_1', g'_2, \ldots, g_n')$ such that for an integer $i,$ $g_i'$ strictly follows $g_i$, and for any integer $j \neq i, g_j'$ is an option of $g_j$ or identical to $g_j$. That is, for an emperor sum of games, a player selects one component and arbitrarily makes numerous moves, but for every other component, the player may move only once.
\end{definition}

Let $\oplus$ be the bitwise XOR operator.

\begin{theorem}
\label{esum}
The position $\mathcal{E}(g_1, g_2, \ldots, g_n)$ is a $\mathcal{P}$-position if and only if $g_i$ is a $\mathcal{P}$-position for any $i$ and $\mathrm{Pl}(g_1) \oplus \mathrm{Pl}(g_2) \oplus \cdots \oplus \mathrm{Pl}(g_n) = 0$.
\end{theorem}

Note that the emperor sum is a special case of the simplicial emperor sum on $\Delta$, such that every face of $\Delta$ has a single vertex.

The remainder of this paper is organized as follows. Section~\ref{sec2} discusses the relationship between {\sc simplicial nim} and impartial games on a simplicial complex. In Section ~\ref{sec3}, we provide a new proof for Theorem \ref{esum}.
The final section presents the conclusions.

\section{Main Result}
\label{sec2}
In this section, we discuss a method for establishing which player has a winning strategy in the simplicial emperor sum.
Let $\Delta$ be a simplicial complex on a set $V = (v_1, v_2, \ldots, v_n )$ and let $P$ be the set of $\mathcal{P}$-positions of {\sc nim} on $\Delta.$

The following lemmas are trivial from the definition of {\sc simplicial nim}.

\begin{lemma}
\label{lem1}

Let $A = (a_1, a_2, \ldots, a_n) \in P.$
For a position $A' = (a'_1, a'_2, \ldots, a'_n),$ if there is a face $F \in \Delta$ such that for any $v_i \in F, a'_i < a_i$ and for any $v_i \not \in F, a'_i = a_i,$ then $A'$ is an $\mathcal{N}$-position.

\end{lemma}
\begin{lemma}
\label{lem2}
For any $B = (b_1, b_2, \ldots, b_n) \not \in P,$ there is a position $B' = (b'_1, b'_2, \ldots, b'_n) \in P$ such that a face $F \in \Delta$ satisfies for any $v_i \in F, b'_i < b_i$, and for any $v_i \not \in F, b'_i = b_i.$ 
\end{lemma}

\begin{theorem}
\label{thm2}
For a collection of positions of impartial games $G = (g_1, g_2, \ldots, g_n),$ $\Delta(G)$ is a $\mathcal{P}$-position, if and only if, $(\mathrm{Pl}(g_1), \mathrm{Pl}(g_2), \ldots, \mathrm{Pl}(g_n)) \in P$, and for any $i$, $g_i$ is a $\mathcal{P}$-position.
\end{theorem}

\begin{proof}
Let $X$ and $Y$ be collections of positions of impartial games such that $X = \{(g_1, g_2, \ldots, g_n): (\mathrm{Pl}(g_1), \mathrm{Pl}(g_2), \ldots, \mathrm{Pl}(g_n)) \in P \text{, and for any } i, g_i \text{ is a } \mathcal{P} \text{-position.}\}$ and $Y = \{(g_1, g_2, \ldots, g_n):(g_1, g_2, \ldots, g_n) \not \in X\}$.

To prove this theorem, it is sufficient to demonstrate the correctness of the following two claims.

\begin{itemize}
    \item [Claim 1.] For any $G \in X,$ every $G'$, such that $\Delta(G')$ is an option of $\Delta(G)$ that satisfies $G' \in Y.$
    
    \item [Claim 2.] For any $H \in Y,$ there is a collection of positions of impartial games $H'$ such that $H' \in X$ and $\Delta(H')$ is an option of $\Delta(H)$.
\end{itemize}

\begin{proof}[Proof of Claim 1]
Assuming that $G \in X$ and $G' = (g_1', g_2', \ldots, g_n'), \Delta(G')$ is an option of $\Delta(G).$
If $G' \in X,$ then $(\mathrm{Pl}(g_1'), \mathrm{Pl}(g_2'), \ldots, \mathrm{Pl}(g_n')) \in P.$ From the assumption $G \in X$, $(\mathrm{Pl}(g_1), \mathrm{Pl}(g_2), \ldots, \mathrm{Pl}(g_n)) \in P.$

Let $F$ be the face chosen in the move $G \rightarrow G'.$ 
If $g_i' \neq g_i,$ then $v_i \in F.$ Therefore, if $\mathrm{Pl}(g_i') < \mathrm{Pl}(g_i),$ then $v_i \in F$ and otherwise, $v_i \not \in F.$ 
Thus, from Lemma \ref{lem1}, in the {\sc nim} on $\Delta,$ $(\mathrm{Pl}(g_1'), \mathrm{Pl}(g_2'), \ldots, \mathrm{Pl}(g_n'))$ is an $\mathcal{N}$-position, which is a contradiction.
\end{proof}

\begin{proof}[Proof of Claim 2]
Assume that $H = (h_1, h_2, \ldots, h_n) \in Y.$ Consider that $H_p = (\mathrm{Pl}(h_1), \allowbreak  \mathrm{Pl}(h_2), \ldots,  \mathrm{Pl}(h_n)).$
Two cases are considered:
\begin{itemize}
\item[(i)]
Every $h_i \in H$ is a $\mathcal{P}$-position:
Because $H \not \in X, H_p$ is not a $\mathcal{P}$-position in {\sc nim} on $\Delta.$
Therefore, from Lemma \ref{lem2}, a position $H_p' = (h_{p1}', h_{p2}', \ldots, h_{pn}') \in P$, and a face $F$ exist such that for every $v_i \in F, h_{pi}' < \mathrm{Pl}(h_i)$ and every $v_i \not \in F, h_{pi}' = \mathrm{Pl}(h_i).$ 
From the definition of the $\mathcal{P}$-position length, for any non-negative integer $m < \mathrm{Pl}(h_i)$, $h_i$ has a strict follower $x$, which is a $\mathcal{P}$-position and $\mathrm{Pl}(x) = m.$
Thus, by choosing $F$, one could transfer from $\Delta(H)$, to $\Delta(H')$, where $H' = (h_1', h_2', \ldots, h_n')$ satisfies $(\mathrm{Pl}(h_1'), \mathrm{Pl}(h_2'), \ldots, \mathrm{Pl}(h_n')) = H_p'$, and every $h_i'$ is a $\mathcal{P}$-position.

\item[(ii)] Otherwise:
Let $h_i^* = h_i$ if $h_i$ is a $\mathcal{P}$-position and otherwise, let $h_i^*$ be an option of $h_i$ and be a $\mathcal{P}$-position.
Consider $H^* = (h_1^*, h_2^*, \ldots, h_n^*)$.

From case (i), there is a collection of positions $H'$, such that $\Delta(H')$ is an option of $\Delta(H^*)$ and $H' \in X.$ 
We have a move $\Delta(H) \rightarrow \Delta(H')$, which is still a legal move in the simplicial emperor sum on $\Delta$ because every component not on a vertex in $F$ is made at most one move.
\end{itemize}

From cases (i) and (ii), for any $H = (h_1, h_2, \ldots, h_n) \in Y, \Delta(H)$ has option $\Delta(H')$ such that $H' \in X.$

\end{proof}

Note that $X$ includes every terminal position $(g_1, g_2, \ldots, g_n),$ where $g_i$ is the terminal position for any $i$. Therefore, from Claims 1 and 2, $X$ is the set of $\mathcal{P}$-positions and $Y$ is the set of $\mathcal{N}$-positions of the simplicial emperor sum on $\Delta$. 
\end{proof}

\section{New aspect of emperor sum of games}
\label{sec3}
The emperor sum of games is a special case of the simplicial emperor sum. Therefore, this sum can be considered within a broader framework. This fact makes the proof of Theorem \ref{esum} very simple.
\begin{proof}[New proof of Theorem \ref{esum}]
Let $\Delta = \bigcup_i (\{\{v_i\}\}).$ The ruleset of
{\sc nim} on $\Delta$ is the same as that of the original {\sc nim}.
Therefore, a $\mathcal{P}$-position $P = (a_1, a_2, \ldots, a_n)$ satisfies $a_1 \oplus a_2 \oplus \cdots \oplus a_n = 0 $.

Further, the emperor sum of $g_1, g_2, \ldots, g_n$ is the same as the simplicial emperor sum on $\Delta.$
Thus, from Theorem \ref{thm2}, a collection of positions of impartial games $G = (g_1, g_2, \ldots, g_n)$ is a $\mathcal{P}$-position of the emperor sum of games, if and only if, ${\rm Pl}(g_1) \oplus {\rm Pl}(g_2) \oplus \cdots \oplus {\rm Pl}(g_n) = 0$, and $g_i$ is a $\mathcal{P}$-position for any $i.$

\end{proof}

\section{Conclusion}
In this study, we consider the simplicial emperor sum and characterize the $\mathcal{P}$-positions for the game.
This result is an extension of the emperor sum and provides a more general picture of the behavior of $\mathcal{P}$-positions when choosing multiple components in one move is allowed.
The study of sums allowing multiple components to be chosen is less advanced than that of disjunctive sums. Therefore, the contributions of this study will aid in future research.

\end{document}